\documentclass{amsart}

\author{Paul Pollack and Noah Snyder}

\usepackage{amssymb,microtype,mathscinet}
\usepackage[utf8]{inputenc}
\address{Department of Mathematics, University of Georgia, Athens, GA 30602}
\email{pollack@uga.edu}

\address{Department of Mathematics, Indiana University, Bloomington, IN 47405}
\email{nsnyder1@indiana.edu}

\DeclareMathAlphabet{\curly}{U}{rsfs}{m}{n}
\newtheorem{thm}{Theorem}

\newtheorem{lem}[thm]{Lemma}
\newtheorem*{claim}{Claim}

\theoremstyle{remark}
\newtheorem*{rmk}{Remark}

\def\O{\mathcal{O}}

\def\I{\mathcal{I}}
\def\Q{\mathbb{Q}}

\def\Z{\mathbb{Z}}
\def\F{\mathbb{F}}

\begin{document}
\title{A Quick Route to Unique Factorization in Quadratic Orders}
\markright{Notes}

\maketitle
\begin{abstract} We give a short proof --- not relying on ideal classes or the geometry of numbers --- of a known criterion for quadratic orders to possess unique factorization.
%This yields arguably the simplest known proof that $\Z[\frac{1+\sqrt{-163}}{2}]$ is a unique factorization domain.
\end{abstract}

\section{Introduction.}
Let $D$ be a quadratic discriminant, meaning that $D$ is a nonsquare integer with $D\equiv 0, 1\pmod{4}$. Set $D = 4d+\sigma$, where $\sigma \in \{0,1\}$, and let $\tau = \frac{\sigma+\sqrt{D}}{2}$. It is easy to check that $\tau^2 \in \Z+\Z\tau$, so that 
\begin{align*} \Z[\tau] &= \Z+\Z\tau\\
&= \left\{\frac{u+v\sqrt{D}}{2}: u, v \in \Z, u\equiv vD\!\!\!\pmod{2}\right\}. \end{align*}
In what follows, we write $\O_D$ (for ``order of discriminant $D$'') in place of $\Z[\tau]$. 
% Let $K/\Q$ be a quadratic extension, so that $K = \Q(\sqrt{d})$ for a uniquely determined squarefree integer $d$. Define
% \[ \O_D = \Z[\tau],\quad\text{where}\quad \tau = \begin{cases} \sqrt{d} &\text{if $d\equiv 2,3\!\!\!\pmod{4}$,}\\
% \frac{1+\sqrt{d}}{2} &\text{if $d\equiv 1\!\!\!\pmod{4}$}.
% \end{cases} \]
% It is easy to check that $\tau^2 \in \Z + \Z\tau$, so that $\O_D = \{a+b\tau: a, b \in \Z\}$. Equivalently, $\O_D = \{a+b\sqrt{d}: a, b \in \Z\}$ when $d\equiv 2,3\pmod{4}$, while $\O_D= \{\frac{1}{2}(a+b\sqrt{d}): a, b \in \Z, a\equiv b\pmod{2}\}$  when $d\equiv 1\pmod{4}$. Set $\Delta = 4d$ or $d$, according to whether $d\equiv 2,3\pmod{4}$ or $d\equiv 1\pmod{4}$, respectively.

Our aim with this note is to showcase a simple proof of the following criterion for unique factorization in $\O_D$. We remind the reader that if $R$ is a domain then $\pi\in R$ is \emph{irreducible} if $\pi$ is nonzero and not a unit, and if whenever $\pi=\alpha\beta$ with $\alpha,\beta \in R$, either $\alpha$ or $\beta$ is a unit. The element $\pi\in R$ is \emph{prime} if $\pi$ is nonzero and not a unit, and if whenever $\pi\mid \alpha\beta$ (with $\alpha,\beta \in R$) either $\pi\mid\alpha$ or $\pi\mid \beta$; equivalently, a prime is a nonzero element of $R$ for which the principal ideal $(\pi)$ is a prime ideal of $R$. Prime elements are always irreducible; the converse holds in a UFD (unique factorization domain), but not in general.

\begin{thm}\label{thm:main} Suppose that every rational prime number
\begin{equation}\label{eq:pbound} p \le \begin{cases} \sqrt{|D|/3} &\text{if $D < 0$,} \\
\sqrt{D/5} &\text{if $D > 0$}
\end{cases} \end{equation}
that is irreducible in $\O_D$ is also prime in $\O_D$. Then $\O_D$ is a unique factorization domain.
\end{thm}

\subsection{Examples.}
\begin{enumerate}
\item[(i)] [$D=73$] Since $\sqrt{73/5} = 3.8\dots$, the conditions of Theorem \ref{thm:main} concern only the primes $p=2$ and $p=3$. Neither 2 nor 3 is irreducible, since
    \[ 2 = \frac{9+\sqrt{73}}{2} \cdot \frac{9-\sqrt{73}}{2}, \quad \text{while} \quad 3 = (2\sqrt{73}+17)\cdot (2\sqrt{73}-17). \]
    (It is easy to check that all of the factors listed here are nonunits.)
    We conclude that $\O_D=\Z[\frac{1+\sqrt{-73}}{2}]$ is a UFD.

    The number $73$ is not particularly special.\footnote{See \cite{PS19} for a counterpoint to this claim.} It is widely believed that there are infinitely many $D>0$ for which $\O_D$ is a UFD. In fact, Cohen and Lenstra have precise conjectures predicting, for instance, that $\O_p$ is a UFD for 75.44\dots\% of primes $p\equiv 1\pmod{4}$ (see \cite[\S5.10]{cohen93} and \cite{CL84b,CL84a,tRW03}).
\item[(ii)] [$D=-163$] Since $\sqrt{163/3} = 7.3\dots$, we must check $p=2, 3, 5, 7$. As $\tau=\frac{1+\sqrt{-163}}{2}$ is a root of the monic irreducible polynomial $X^2-X+41$, we have that $\Z[\tau] \cong\Z[x]/(X^2-X+41)$. Hence, for each prime $p$,
    \[ \Z[\tau]/(p)  \cong (\Z[X]/(p))/(X^2-X+41) \cong \F_p[x]/(X^2-X+41). \]
     It is straightforward to check that $X^2-X+41$ is irreducible modulo $p$ for each  of $p=2,3,5,7$. (For the odd primes $p$ in this list, it suffices to observe that the discriminant $-163$ of $X^2-X+41$ is a nonsquare mod $p$.) Therefore, $\Z[\tau]/(p)$ is a field, whence $(p)$ is a prime ideal of $\O_D$ and $p$ is a prime element. So  the criterion of Theorem \ref{thm:main} is again satisfied and  $\O_D$ is a UFD. The number $-163$ \emph{is} special; as shown by Heegner, it is the largest (in absolute value) negative $D$ for which $\O_D$ is a UFD (\cite{heegner52}; see also \cite{cox13}).
 \end{enumerate}

We do not claim that Theorem \ref{thm:main} is new. When $\O_D$ is the full collection of algebraic integers inside $\Q(\sqrt{D})$ (the so-called ``maximal order''), basic algebraic number theory says that $\O_D$ is a Dedekind domain with finite class group. Furthermore, results from the geometry of numbers imply that every ideal class is represented by an ideal with norm  bounded by the quantities appearing on the right of \eqref{eq:pbound} (see \cite[Theorem 13.7.10, p.\ 399]{artin11} for $D < 0$ and \cite[Exercise 17, p.\ 300]{cohen93} for $D > 0$). So Theorem \ref{thm:main} follows easily (in this case).

It seems of some interest --- e.g., for the teaching of basic courses in algebra and number theory --- to give a proof of Theorem \ref{thm:main} requiring as little machinery as possible. Several close relatives of Theorem \ref{thm:main} have been proved in the literature without reference to algebraic number theory; see \cite{fendel85,gyarmati83,lanczi65,
popovici57b,popovici57,ramirez16,snyder07,zaupper83,zaupper90}. However, all of these papers either establish results weaker or less complete than  Theorem \ref{thm:main}, or their proofs depend on auxiliary results from the geometry of numbers or the theory of Diophantine approximation. (For example, the beautifully simple method of Ramirez V.\ in \cite{ramirez16} gives a very satisfactory result when $D<0$, but only a partial result for $D>0$.)  Apart from a few easy lemmas concerning the ``norm'' map (see the Notation section below), our proof of Theorem \ref{thm:main} is self-contained, resting only on the commutative ring theory seen in a first graduate algebra course. 

% Zaupper \cite{zaupper90} gave a proof of Theorem \ref{thm:main} independent of the major theorems of algebraic number theory. A similar argument (for a slightly weaker conclusion) was sketched by Snyder \cite{snyder07}. Both of their arguments invoke results from the geometry of numbers.
%
%Proofs not dependent on algebraic number theory or the geometry of numbers, for weaker statements closely related to Theorem \ref{thm:main}, were published already by Popovici in 1957 (see \cite[Theorems 18, 19]{popovici57}). Popovici's paper is not easy for most modern mathematicians to read. In addition to having been published only in Russian, both the terminology and the arrangement of the arguments seem awkward and inefficient by present-day standards. The aim of this note is to give a simple and direct proof of Theorem \ref{thm:main}, in the spirit of \cite{popovici57} but with significant simplifications.
%%Instead, the 'engine' of the argument is a descent of a kind that would be instantly recognizable to Fermat.
%
%\begin{rmk} The reader interested in elementary approaches to unique factorization in the rings $\O_D$ should also consult \cite{popovici57}, \cite{lanczi65, gyarmati83}, \cite{zaupper83}, \cite{fendel85}, \cite{campoli88,LdO18}.
%\end{rmk}

\subsection*{Notation.} We let $K$ be the fraction field of $\O_D$, so that $K = \Q(\sqrt{D})$, and we denote conjugation in $K$ with a bar. The \emph{norm} of $\alpha\in K$, denoted $N(\alpha)$, is defined by $N(\alpha) = \alpha\bar{\alpha}$. We recall that $N(\alpha\beta) = N(\alpha)N(\beta)$ for all $\alpha,\beta \in K$, that the norm sends nonzero elements of $\O_D$ to nonzero integers, and that $\alpha$ is a unit of $\O_D$ if and only if $N(\alpha)=\pm 1$. Readers are invited to prove these  results themselves; alternatively, they may consult, e.g., \cite[Chapter 2]{lehman19}.

\section{Proof of Theorem \ref{thm:main}.}

Our proof makes crucial use of the following lemma, which also features in the arguments of \cite{ gyarmati83,popovici57,ramirez16,snyder07,zaupper83,zaupper90}.

\begin{lem}\label{lem:primenorm} Let $\alpha \in \O_D$. If $N(\alpha)=\pm p$, where $p$ is a rational prime, then $\alpha$ is prime in $\O_D$.
\end{lem}

\begin{proof} 
% It can be shown (see, e.g., Corollary 6.10 on p.\ 54 of \cite{pollack17}) that for every nonzero $\alpha \in \O_D$, \begin{equation}\label{eq:normequalssize} \#\O_D/(\alpha) = |N(\alpha)|.\end{equation} Thus, under the present hypotheses, $\#\O_D/(\alpha) = p$, making $\O_D/(\alpha)$ isomorphic to the field $\F_p$. So the ideal $(\alpha)$ is prime (in fact, maximal), and $\alpha$ is a prime of $\O_D$.

% For the sake of completeness, we give a simple proof of \eqref{eq:normequalssize} under the hypothesis of the lemma.

Since $\alpha \bar{\alpha}= \pm p$, there is a canonical surjection $\O_D/(p) \twoheadrightarrow \O_D/(\alpha)$. Since $\bar{\alpha}$ is not a unit, the  corresponding kernel is nontrivial (containing, e.g., $\alpha\bmod{p}$). Thus, $\#\O_D/(\alpha)$ is a proper divisor of $\#\O_D/(p) = p^2$. (The last equality comes from noting that $a+b\tau$, for $0\le a, b < p$, form a complete residue system mod $p$.) Since $\alpha$ is not a unit, $\#\O_D/(\alpha) > 1$. Therefore, $\#\O_D/(\alpha)=p$, and so $\O_D/(\alpha) \cong \F_p$. Hence, $(\alpha)$ is a prime (in fact, maximal) ideal of $\O_D$, so that $\alpha$ is prime in $\O_D$.
\end{proof}

We turn now to the proof of Theorem \ref{thm:main}. A simple induction on $|N(\alpha)|$ shows that every nonzero, nonunit $\alpha \in \O_D$ has a factorization into irreducibles. So it remains only to prove uniqueness. We  reduce this (as in \cite{popovici57, snyder07,zaupper90}) to  the following claim.

\begin{claim} Every prime in $\Z$ factors as a product of primes in $\O_D$.
\end{claim}

To see why this suffices, recall that an element with a factorization into primes necessarily has this as its only factorization into irreducibles (up to order and unit factors). This is clear from the usual proof of unique factorization in  a Euclidean domain or PID (compare with the proof of  Proposition 12.2.14(a) in \cite{artin11}). Since every rational integer larger than $1$ factors as a product of rational primes, our claim implies that all those integers factor uniquely in $\O_D$. But this implies that every $\alpha \in \O_D$, not zero and not a unit, also factors uniquely: If $\alpha$ had two factorizations, we could cook up two factorizations of $|N\alpha| = \pm \alpha\bar{\alpha}$ by concatenating our factorizations of $\alpha$ with a fixed factorization of $\pm \bar{\alpha}$.

% To prove the claim, we take cases according to whether $d>0$ or $d<0$.

\begin{proof}[Proof of the claim]  Assuming the claim to be false, let $p$ be the smallest prime for which it fails. Then 
\begin{equation}\label{eq:plower} p >
\begin{cases} \sqrt{|D|/3} &\text{when $D<0$}, \\
\sqrt{D/5} &\text{when $D>0$}.
\end{cases}
\end{equation}
Indeed, suppose otherwise. Since $p$ does not factor as a product of primes, it itself is not prime. But then the hypothesis of Theorem \ref{thm:main} tells us that $p$ factors nontrivially in $\O_D$. Write $p=\pi_1\cdots\pi_k$, with $k\ge2$ and all the $\pi_i$ irreducible. Taking norms, $p^2=N(\pi_1)\cdots N(\pi_k)$, and so $k=2$ and $N(\pi_1) = N(\pi_2) = \pm p$. By Lemma \ref{lem:primenorm}, both $\pi_1$ and $\pi_2$ are prime, and so $p$ factors into primes after all, an absurdity.

Let \[ m(X) = X^2 - \sigma X + \frac{\sigma-D}{4} \in \Z[X] \]
be the minimal polynomial of $\tau$. Then $\Z[\tau]\cong \Z[X]/(m(X))$
 and $\Z[\tau]/(p) \cong \F_p[X]/(m(X))$. Since $p$ is not prime in $\O_D$, the quotient ring $\Z[\tau]/(p)$ is not a field, and so $m(X)$ factors nontrivially over $\F_p$. Thus, for some integers $x$ and $x'$,
 \begin{equation}\label{eq:mfact} m(X) \equiv (X-x) (X-x') \pmod{p}. \end{equation}
 Comparing coefficients of $X$ on both sides, we find that $x+x' \equiv \sigma\pmod{p}$, and so we can assume that
 \[ \frac{p+\sigma}{2} \le x\le p \text{\quad if $D>0$}, \qquad\text{and}\qquad \sigma\le x \le \frac{p+\sigma}{2} \text{\quad if $D<0$}. \]
By \eqref{eq:mfact}, $m(x) \equiv 0\pmod{p}$. Moreover, our inequalities for  $x$ guarantee that
\[ |m(x)| < p^2. \]
Indeed, if $D>0$, then (keeping in mind \eqref{eq:plower})
\[ p^2 > m(x) = \left(x-\frac{\sigma}{2}\right)^2 - \frac{D}{4} \ge \frac{p^2-D}{4} > -p^2, \]
while if $D < 0$, then
\[ 0 < m(x) = \left(x-\frac{\sigma}{2}\right)^2 + \frac{|D|}{4} \le \frac{p^2+|D|}{4} < p^2. \]
Write $m(x)=pr$, where $|r| < p$. By the minimality of $p$, every prime dividing $r$ factors into primes of $\O_D$, and so $r$ itself factors, up to sign, as a product of primes of $\O_D$. Thus, for some primes $\eta_1,\dots,\eta_\ell$ of $\O_D$,
\[ (x-\tau)(x-\bar{\tau}) = m(x) = \pm p \eta_1\cdots\eta_\ell. \]
Since $\eta_1$ is prime, $\eta_1$ divides either $x-\tau$ or $x-\bar{\tau}$. Divide both sides of the equation by $\eta_1$ and continue the process with $\eta_2$. Eventually we are led to a factorization of the form
\[ \frac{x-\tau}{\Pi_1} \cdot \frac{x-\bar{\tau}}{\Pi_2} = \pm p, \qquad\text{where}\qquad
\Pi_1 = \prod_{i \in \I} \eta_i,\quad \Pi_2 = \prod_{i \in \I^{c}} \eta_i
\]
 for some $\I \subset \{1,2,\dots,k\}$, where $\I^{c} = \{1,2,\dots,k\}\setminus \I$. Multiplying by $\pm 1$ if necessary, we obtain a factorization of $p$ as $\alpha\beta$, say. If $\alpha$ or $\beta$ is a unit, then the other is a unit multiple of $p$. But that implies  $p\mid x-\tau$ or $p\mid x-\bar{\tau}$, which is absurd. (Both $\{1,\tau\}$ and $\{1,\bar{\tau}\}$ are $\Z$-module bases of $\O_D$, and so when a multiple of $p$ is written as $a+b\tau$ or $a+b\bar{\tau}$, both $a$ and $b$ must be multiples of $p$.) So $\alpha,\beta$ are nonunits. Now taking norms shows that $N\alpha = N\beta=\pm p$, so that $\alpha,\beta$ are prime by Lemma \ref{lem:primenorm}. Thus, $p$ has a factorization into primes of $\O_D$ after all, contradicting the choice of $p$. 
\end{proof}

\begin{rmk} In 1912/1913, Frobenius \cite{frobenius12} and Rabinowitsch \cite{rabinowitsch13} (independently) published the following striking result: For each integer $q \ge 2$, 
\begin{multline*} x^2-x+q~\text{ is prime for all integers $0 < x < q$}\\\text{if and only if}\quad \text{$\Z[\tfrac{1}{2}(1+\sqrt{1-4q})]$ is a UFD}; \end{multline*} 
see \cite[Chapter 11]{pollack17} for an exposition. For example, since $\Z[\frac{1}{2}(1+\sqrt{-163})]$ is a UFD, the polynomial $x^2-x+41$ assumes prime values for $x=1,2,\dots,40$. The  ``only if'' half of the proof is the more difficult of the two, and for this most modern treatments fall back on the theory of the class group. Theorem \ref{thm:main} allows one to fashion a completely elementary proof (apply Theorem \ref{thm:main} in place of Proposition 11.13 in \cite{pollack17}; alternatively, Ramirez V.'s Theorem 3.1 from \cite{ramirez16} can be used).  Indeed, these arguments prove a sharper version of the forward direction, which has the following consequence: $x^2-x+41$ being prime for just $x=1,2,3,4$ implies that $x^2-x+41$ must continue being prime all the way to $x=40$. Certain relatives of Rabinowitsch's theorem for real quadratic orders can  be given elementary proofs in a parallel way (compare with \cite{ramirez19}).
\end{rmk}

\subsection*{Acknowledgements}
The authors are supported by the National Science Foundation (NSF) under awards DMS-2001581 (P.\,P.) and DMS-1454767/DMS-2000093 (N.\,S.). They thank Enrique Trevi\~no and the referees for helpful suggestions. In particular, they are grateful to a referee for pointing out that the argument applies for orders other than the maximal one.

\end{document}